\documentclass[11pt]{article} 
 \pdfoutput=1
\usepackage{amsmath,amsthm}
\usepackage{amssymb}
\usepackage[authoryear]{natbib}
\usepackage{tikz}
\usepackage[linguistics]{forest}
\usepackage[mathscr]{eucal}
\usepackage{verbatim}
\usetikzlibrary{matrix,fit,backgrounds}
\usetikzlibrary{decorations.pathreplacing}

\theoremstyle{plain}
\newtheorem{theorem}{Theorem}
\newtheorem{proposition}{Proposition}
\newtheorem{corollary}{Corollary}
\newtheorem{lemma}{Lemma}

\theoremstyle{remark}

\title{Instability of LIFO Queueing Networks}
\author{Maury Bramson, \\
University of Minnesota\\{bramson@math.umn.edu}}

\begin{document}
\title{Instability of LIFO Queueing Networks}
\author{Maury Bramson \\
University of Minnesota\\{bramson@math.umn.edu}}

\maketitle
	\begin{abstract}  
Under the last-in, first-out (LIFO) discipline, jobs arriving later at a class 
always receive priority of service over
earlier arrivals at any class belonging to the same station.  Subcritical LIFO queueing networks with Poisson external arrivals are known to be stable, but an open problem has been whether this is also the case 
when external arrivals are given by renewal processes.  Here, we show that
this weaker assumption is not sufficient for stability by constructing a family of
examples where the number of jobs in the network increases to infinity over time.

This behavior contrasts with that for the other classical disciplines:  processor sharing (PS), 
infinite server (IS), and first-in, first-out (FIFO), which are stable under general 
conditions on the renewals of external arrivals.  Together with 
LIFO, PS and IS constitute the classical symmetric disciplines; with the
  inclusion of
FIFO, these disciplines constitute the classical homogeneous disciplines.  Our examples show that a
general theory for stability of either family is doubtful.
\end{abstract}

\section{Introduction} 
\label{intro}

Under the preemptive last-in, first-out (LIFO) discipline (or policy), jobs in a queueing network 
arriving at a class always receive priority of service over earlier arrivals at any class
belonging to the same station.  Service for the preempted
jobs continues after later-arriving jobs have been
served.  This rule is quite
natural, and corresponds to later occurring tasks always being given priority over earlier
ones, for instance, new jobs being given priority in a piled stack of work to be done.
LIFO
is a GAAP accepted accounting method for inventory.

The LIFO discipline is one of the four ``classical"
disciplines that were analyzed in the famous papers \cite{BCMP75} and \cite{Ke75, Ke76}, the other 
disciplines being processor sharing (PS), infinite server (IS), and first-in, first-out (FIFO).
In these papers, the stability of these four queueing networks 
was shown when the Poisson input is subcritical, that is, the corresponding
Markov processes are
positive recurrent given that work on the average arrives at a slower rate than 
it would be served if all servers are fully active when there are jobs in the network.  

Since these papers,
substantial progress has been made in showing the stability of subcritical queueing networks under 
the PS, IS, and FIFO disciplines when the input is generalized from Poisson to 
renewal processes.  However, little is currently known about the stability of the 
LIFO discipline in the non-Poisson setting.  In this 
paper, we demonstrate instability for a family of subcritical LIFO queueing networks by
showing that the number of jobs in the network increases to infinity over time.  

To define this family, we first give the
network topology and then its external 
arrival and service processes.  The network consists of
four stations, with a total
of six classes, and  is pictured in
Figure \ref{fig1}.  Jobs enter the network at either Class 1 or Class 4.  
The jobs arriving at Class 1 are routed successively
through Classes 2 and 3, before leaving the network, and the jobs arriving at 
Class 4 are routed successively through Classes
5 and 6 before leaving.  Classes 1 and 6 together comprise Station I, 
Classes 3 and 4 together comprise Station IV, and
Classes 2 and 5 each form their own single-class stations, Stations II and III.  
Except for the presence of Classes 2
and 5, the network has the same structure as the well-known Rybko-Stolyar network.

The external arrival and service processes are each symmetrically defined, with jobs entering the upper route following the same rules
as those entering the lower route.  For each of the two routes, external arrivals are given by independent renewal processes, whose interarrival
times are i.i.d. random variables with measure $\nu$ given by
\begin{alignat}{1}
\label{entrancelaws}
&\nu(dt) = \tfrac{1}{M} \text{e}^{-\beta(t - \gamma M)}dt  \quad \text{for } t\in [\gamma M,2M], \\
\notag & \nu(\{\tfrac{1}{M^2}\}) = 1 - \tfrac{1}{M}, 
\end{alignat}

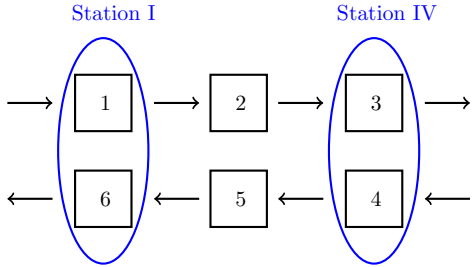
\begin{figure} 
\centering
\begin{tikzpicture}[thick, scale=0.75, every node/.style={transform shape}]
\draw[thick, ->] (-2.2,1) -- (-1.4,1); 
\draw (-1,0.5) rectangle (0,1.5);
\draw[thick,->] (0.4,1) -- (1.2,1); 
\draw (1.4,0.5) rectangle (2.4,1.5);
\draw[thick,->] (2.6,1) -- (3.4,1); 
\draw (3.8,0.5) rectangle (4.8,1.5);
\draw[thick,->] (5.2,1) -- (6.0,1); 
\draw[blue] (-.5,0.15) ellipse (.8 and 2);
\draw[blue](4.3,0.15) ellipse (.8 and 2);
\draw[thick, ->] (-1.4,-.7) -- (-2.2,-.7); 
\draw (-1,-1.2) rectangle (0,-.2);
\draw[thick,->] (1.2,-.7) -- (0.4,-.7); 
\draw (1.4,-1.2) rectangle (2.4,-.2);
\draw[thick,->] (3.4,-.7) -- (2.6,-.7); 
\draw (3.8,-1.2) rectangle (4.8,-.2);
\draw[thick,->] (6.0,-.7) -- (5.2,-.7); 

 \filldraw[black] (-.70,1)  node [anchor=west] {1};
 \filldraw[black] (1.7,1)  node [anchor=west] {2};
\filldraw[black] (4.1,1)  node [anchor=west] {3};
 \filldraw[black] (-.70,-.7)  node [anchor=west] {6};
\filldraw[black] (1.7,-.7)  node [anchor=west] {5};
\filldraw[black] (4.1,-.7)  node [anchor=west] {4};

\filldraw[blue] (-1.20,2.6)  node [anchor=west] {Station I};
\filldraw[blue] (3.5,2.6)  node [anchor=west] {Station IV};
\end{tikzpicture}
\caption{Squares in top row correspond to Classes 1-3, in order of appearance along their route; similarly, 
squares in lower row correspond to Classes 4-6, in order of appearance.  Classes in left oval belong
to Station I, classes in right oval belong to Station IV; Classes 2 and 5 belong to the one-class stations, Stations II and III.}
\label{fig1}
\end{figure}
\noindent where $M$ is assumed to be large, and $\beta$ and $\gamma$ are chosen so that 
$\nu$ has both measure and mean 1.  (One has $\beta > 1$, $\gamma <1$, with $\beta \sim 1, \gamma \sim 1$ for large $M$.)

The service laws at Classes 1, 3, 4, and 6 are all deterministic, whereas the service laws at Classes 2 and 5 are exponentially
distributed, with the means at different classes being given by
\begin{equation}
\label{means}
m_1 = m_4 = \delta^3, \quad m_2 = m_5 = 1 - \delta,  \quad m_3 = m_6 = 1 -\delta + \delta^3,
\end{equation}
where $\delta = 1/M^{1/15}$.  We assume that all interarrival and service times are independent of each other.  Since, for small $\delta$,
\begin{equation}
\label{subcritical}
m_1 + m_6 = m_3 + m_4 = 1 - \delta + 2\delta^3 < 1, \quad m_2 = m_5 = 1 -\delta <1,
\end{equation}
 the system is subcritical.

The system is LIFO, with jobs entering a given class always receiving priority of service over
earlier arrivals at any class of its station.  (In case of a ``tie", either priority is allowed.)  We assume that the network is preemptive resume, with jobs currently in service
being interrupted by arrivals, and continuing their service in the absence of more recent arrivals.  Jobs originally in
the network are assigned an arbitrary ordering for service at their class.

Denoting by $Z(t)$ the total number of jobs in the network at time $t$, Theorem \ref{theorem1} asserts that $Z(t) \rightarrow\infty$ as $t \rightarrow\infty$.
\begin{theorem}
\label{theorem1}
Suppose $M$ is sufficiently large.  For any LIFO queueing network with routing as 
in Figure \ref{fig1}, and external arrival and service processes
as in (\ref{entrancelaws})-(\ref{means}),
\begin{equation}
\label{firstlimit}
Z(t)\rightarrow\infty  \quad \text{ almost surely as } t\rightarrow \infty.
\end{equation}
\end{theorem}

\noindent The analog of (\ref{firstlimit}) holds for the total  work
in the network. We comment on this immediately
after the proof of Theorem  \ref{theorembehavioratS}.

When $L = (1-\delta +\delta^3)/\delta^3$ is an integer, one can create a
second family of queueing networks by partitioning
the Classes 3 and 6 into 
$L$
new classes each, Classes 3.1,...,3.L and 6.1,...,6.L,  creating 
in this manner new Stations I and IV that
each have $L+1$ classes (see Figure
\ref{fig2}).  By
immediately continuing service at 
Classes 3.($\ell$ +1) and 6.($\ell$+1)
 for jobs
departing from Classes 3.$\ell$ and 6.$\ell$,
the new
queueing networks thus defined will also have the LIFO discipline. 

The external arrival processes of this second family 
are again defined as in (\ref{entrancelaws}).  The service processes for Classes 1, 2, 4, and 5 are
as in the first family, and have means given by (\ref{means}).  The service times for Classes 3.1 through 3.L and 6.1 through 6.L
are deterministic, and satisfy
\begin{equation}
\label{modifiedmeans}
m_{3.1} = \ldots = m_{3.L} = m_{6.1} = \ldots = m_{6.L} = \delta^3.
\end{equation}
Since the mean service times of each of the classes
at Stations I and IV is equal to $\delta^3$, Stations I and IV are of \emph{Kelly type}, that is, the mean service times of the classes at the
station are equal.  The following analog of (\ref{subcritical}) holds,
\begin{equation}
\label{modifiedsubcritical}
m_1 + \sum_{\ell = 1}^L m_{6.\ell}  = m_4 + \sum_{\ell = 1}^L m_{3.\ell}
= 1 - \delta + 2\delta^3 < 1, \quad m_2 = m_5 = 1 -\delta <1,
\end{equation}
and so the system is subcritical.

Corollary \ref{corollary1} therefore immediately follows
from Theorem \ref{theorem1}.

\begin{corollary}
\label{corollary1}
Suppose $M$ is sufficiently large.  For any LIFO queueing network with routing as 
in Figure \ref{fig2}, external arrival processes as in (\ref{entrancelaws}), and service processes
for Classes 1, 2, 4, and 5 as in (\ref{means}) and Classes 3.$\ell$ and 6.$\ell$ as in (\ref{modifiedmeans}),
\begin{equation}
\label{firstlimitkelly}
Z(t)\rightarrow\infty  \quad \text{ almost surely as } t\rightarrow \infty.
\end{equation}
\end{corollary}
\begin{figure}
\centering
\begin{tikzpicture}[thick,scale=1.00, every node/.style={scale=0.70}]
\draw[thick,->] (-3.7,1) -- (-1.4,1); 
\draw (-1.2,0.5) rectangle (-0.8,1.5);
\draw[thick,->] (-0.6,1) -- (1,1); 
\draw (1.20,0.5) rectangle (2.10,1.5);
\draw[thick,->] (2.3,1) -- (2.95,1); 
\draw (3.0,0.5) rectangle (3.4,1.5);
\draw[thick,->] (3.45,1) -- (3.73,1); 
\draw (3.78,0.5) rectangle (4.18,1.5);
\draw[thick,->] (4.22,1) -- (4.50,1); 
\draw (4.61,1) ellipse (0.015 and 0.015);
\draw (4.69,1) ellipse (0.015 and 0.015);
\draw (4.77,1) ellipse (0.015 and 0.015);
\draw[thick,->] (4.88,1) -- (5.16,1); 
\draw (5.2,0.5) rectangle (5.6,1.5);
\draw[thick,->] (5.7,1) -- (6.9,1); 
\draw[blue] (-1,0.15) ellipse (2 and 2);
\draw[blue] (4.3,0.15) ellipse (2 and 2);
\draw[thick,->] (-2.4,-.7) -- (-3.7,-.7); 
\draw (-2.28,-1.2) rectangle (-1.88,-.2);
\draw[thick,->] (-1.56,-.7) -- (-1.84,-.7); 
\draw (-1.45,-.7) ellipse (0.015 and 0.015);
\draw (-1.37,-.7) ellipse (0.015 and 0.015);
\draw (-1.29,-.7) ellipse (0.015 and 0.015);
\draw[thick,->] (-.90,-.7) -- (-1.18,-.7); 
\draw (-0.86,-1.2) rectangle (-0.46,-.2);
\draw[thick,->] (-.14,-.7) -- (-0.42,-.7); 
\draw (-0.06,-1.2) rectangle (0.34,-.2);
\draw[thick,->] (1,-.7) -- (0.4,-.7); 
\draw (1.20,-1.2) rectangle (2.10,-.2);
\draw[thick,->] (3.9,-.7) -- (2.3,-.7); 
\draw (4.1,-1.2) rectangle (4.5,-.2);
\draw[thick,->] (6.9,-.7) -- (4.7,-.7); 
\filldraw[black] (-1.16,1.0)  node [anchor=west] {1};
\filldraw[black] (1.50,1.0)  node [anchor=west] {2};
\filldraw[black] (2.94,1.0)  node [anchor=west] {3.1};
\filldraw[black] (3.71,1.0)  node [anchor=west] {3.2};
\filldraw[black] (5.11,1.0)  node [anchor=west] {3.L};
\filldraw[black] (4.138,-.7)  node [anchor=west] {4};
\filldraw[black] (1.47,-.7)  node [anchor=west] {5};
\filldraw[black] (-0.11,-.7)  node [anchor=west] {6.1};
\filldraw[black] (-0.93,-.7)  node [anchor=west] {6.2};
\filldraw[black] (-2.35,-.7)  node [anchor=west] {6.L};
\filldraw[blue] (-1.5,2.6)  node [anchor=west] {Station I};
\filldraw[blue] (3.6,2.6)  node [anchor=west] {Station IV};
\end{tikzpicture}
\caption{Squares in top row correspond to Classes 1, 2, and 3.1-3.L, in order of appearance along their route; similarly, 
squares in lower row correspond to Classes 4, 5, and 6.1-6L, in order of appearance.  Classes in left circle belong
to Station I, classes in right circle belong to Station IV; Class 2 and Class 5 belong to the one-class stations, Stations II and
 III.  All classes at a given station have the same service rule.}
\label{fig2}
\end{figure}
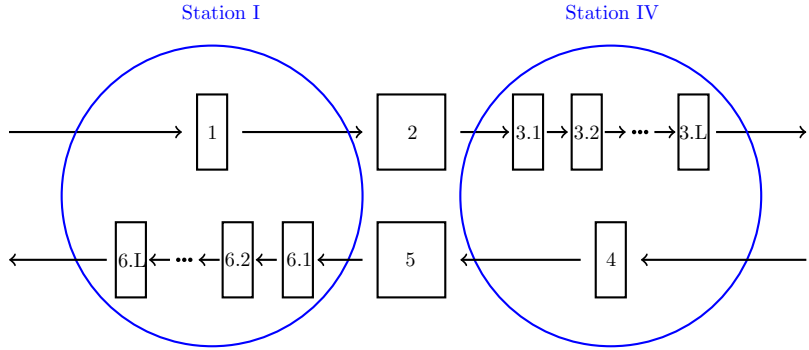

\subsection{Historical context and some philosophy}  
\label{sechistorical}
The evolution of the theory of multiclass queueing networks has been strongly
influenced by explicit results for the four ``classical disciplines", 
PS, LIFO, IS, and FIFO.  The first three of these disciplines are \emph{symmetric} disciplines,
whereas FIFO is a member of the more general family of\emph{ homogeneous} disciplines.  
The distinguishing property for homogeneous 
disciplines is that the distribution of the ordered
position assigned to a job arriving 
at a station and the fraction of
service assigned to a job based on its position both do not
depend on the class of the job within that station.  For symmetric disciplines, the assigned arrival and service 
distributions are 
equal to one another at each station.  
(See the original sources \cite{BCMP75} 
and \cite{Ke75, Ke76, Ke79}, or the monograph \cite{Br08}, for complete
definitions.)

Subcritical networks with a symmetric discipline and Poisson external
arrivals are stable irrespective of the distributions of service times at individual
classes, and the corresponding Markov processes are positive recurrent with
equilibria (i.e., stationary distributions) that have an explicit product form.  Subcritical networks with
a homogeneous discipline, Poisson external arrivals, and exponentially distributed
service times that have the same mean at a given station are also positive recurrent, with
equilibria that have a similar product form.  These two families of disciplines are among
the few disciplines for which the equilibria of multiclass queueing networks are explicitly 
computable.

These results contributed to
overly rosy expectations for the stability of subcritical queueing networks for arbitrary work-conserving
disciplines, even though their equilibria were not expected to be explicitly calculable.  
However, examples in various settings later showed stability need not follow from subcriticality
(see, e.g., \cite{Br94}, \cite{LuKu91}, \cite{RySt92}, \cite{Se94}), including for FIFO networks 
whose classes at a given station have
unequal mean service times.

On the other hand, in  \cite{Da95} and \cite{RySt92}, a machinery
was developed that enabled one to show stability of queueing networks in a wide range of settings,
where external arrivals were allowed to consist of renewal rather than Poisson processes, and no assumptions on the service times of classes were needed.  In one such application,
subcritical FIFO networks, with classes at a given station having the same
mean service time, were shown to be stable (\cite{Br96a}).  
Stability for subcritical PS networks can be shown in a similar manner, and is discussed
in the appendix.
Because of the presence of
an infinite number of available servers, IS queueing networks are stable in
all settings.  Little is currently known about the
stability of subcritical multiclass queueing networks
with the LIFO discipline.

Theorem \ref{theorem1} of this paper shows that 
subcritical queueing networks with the
LIFO discipline need not be stable.   A 
consequence is the absence of a uniform framework
for establishing the stability of symmetric disciplines, in contrast
to when input is Poisson.
On account of Corollary \ref{corollary1}, a subcritical LIFO
network being of Kelly type is also not sufficient for
stability.  So there is no uniform framework
for establishing the stability of homogeneous 
disciplines, again in contrast to the Poisson case.

\subsection{Overview of the paper}
\label{overview}

In Section \ref{sectionstatespace}, we construct the state space.  Since one needs to
keep track of residual times for all jobs except those at Classes 2 and 5, where one
does not wish to know the residual times, the
state space is somewhat nonstandard.

The demonstration of Theorem \ref{theorem1} involves the construction of ``cycles". 
One shows, at the end of each cycle, that the state of the process is typically an approximate multiple
of the state at the beginning of the cycle, except that the roles of Classes 1-3 and
Classes 4-6 have been switched. In Section \ref{sectioninductionstep}, this induction
step, Theorem \ref{propinductionstep}, is stated, and Theorem \ref{theorem1} is
demonstrated using Theorem \ref{propinductionstep}.

In Section \ref{secintuition}, we provide heuristics for the proof of 
Theorem \ref{propinductionstep}, avoiding the technical details.   
We also remark on the behavior of the queueing network
under certain modifications.  The content in this section
is optional, but may be helpful to the reader.

Theorem  \ref{propinductionstep} is demonstrated  by dividing each cycle into two
random time intervals, $[0,S_1]$ and $[S_1,S_1+S_2]$.  The behavior of the network during the much longer time interval $[0,S_1]$ is
analyzed in 
Section \ref{sectionuptoS}, where Theorem \ref{theorembehavioratS} is
demonstrated; most of the technical work in this paper is devoted to showing Theorem \ref{theorembehavioratS}.
 The behavior of the network during the much shorter time interval  $[S_1, S_1+S_2]$ is
analyzed in Section \ref{proofoftheorem}.

In Subsection \ref{sechistorical}, we were somewhat vague about known stability results
for PS queueing networks. In the appendix, we show
stability on a dense set of service time distributions for subcritical PS queueing networks 
whose external arrivals are given by renewal processes.  The result follows 
quickly from results on the stability of subcritical HLPPS queueing networks 
in  \cite{Br96b}. 
An extension to
all service times does not follow in an obvious manner.

\section{State space construction}
\label{sectionstatespace}

In this section, we construct the state space $\mathscr{S}$ of the Markov process
corresponding to the family of LIFO queueing networks in Figure \ref{fig1}.  The space $\mathscr{S}$
consists of points $x$ of the form
\begin{equation}
\label{eqss}
x \in (\mathbb{Z} \times \bar{\mathbb{R}} \times \mathbb{R})^{\infty} \times \mathbb{Z}^2\times \mathbb{R}^2,
\end{equation}
subject to appropriate positivity conditions 
($\bar{\mathbb{R}} := \mathbb{R} \cup \{ \infty\}$).
Only a finite number of the
coordinates of $(\mathbb{Z} \times \bar{\mathbb{R}} \times \mathbb{R})^{\infty}$, indexed by $i$, are assumed to be nonzero.  For
each such nonzero triple, the first coordinate $k_i$ is to be interpreted as the current class of a
job in the network, selected from among the classes $k_i = 1, 3, 4,  6$.   The second coordinate $s_i$ 
measures how long ago such a job entered the class;  we set $s_i = \infty$ for
a job $i$ originally at the class.
The third coordinate $v_i$ measures the residual service
time for a job at the  class.  The first coordinate $k_i$
is given in descending order, followed by $s_i$, also in
descending order.  The coordinates $z_2$ and $z_5$ of $\mathbb{Z}^2$ 
are to be interpreted as the number of jobs at 
the two remaining classes, Class 2 and Class 
5; since they comprise single class stations, it is
not necessary to keep track of arrival times
of jobs, and since we 
wish to preserve the memoriless property of the 
exponentially distributed service times, we do not
include the residual times of jobs in the state space
descriptor. The coordinates  $u_1$ and $u_4$
of $\mathbb{R}^2$ are the residual interarrival times at
Classes 1 and 4.

 We equip the state space $\mathscr{S}$ with the metric
\begin{align}
\label{eqmetric}
d(x,x') \, = \, \sum_{i=1}^{\infty} 
\big((|k_i - k_i'| & +  |s_i - s_i'| + |v_i - v_i'|) \wedge 1 \big) \\
 + \sum_{i=1}^2 & |z_{a_i} - z_{a_i}'| 
+ \sum_{i=1}^2 |u_{b_i} - u_{b_i}'|, \notag
\end{align}
where $a_1 = 2$, $a_2 = 5$, $b_1 = 1$, and $b_2 = 4$.
We denote by $\mathfrak{S}$ the standard Borel 
$\sigma$-algebra inherited from the metric.

The Markov process underlying the LIFO queueing
network in Figure \ref{fig1} is defined to be the stochastic process
$X(t)$, $t \ge 0$, whose state at any time is given by a
point $x_t \in \mathscr{S}$ that evolves according to the LIFO
rule; the accompanying filtration $\mathfrak{F}_t$ is defined in the usual manner.
Although this Markov process is not Feller, it is strong
Markov.  This is not immediate obvious; for more
detail on the construction of the Markov process
and its strong Markov property, see 
\cite{Br08}, Chapter 4.5.  (In our present setting,
the definition of coordinates in (\ref{eqss}) is slightly
different.)

We denote by $z_k$, $k=1,\ldots, 6$, the number of 
jobs in each class and by $z= \sum_{k=1}^6 z_k$ the
number of jobs in the network.  ($z_2$ and $z_5$ are employed
in (\ref{eqmetric}).)  Denote by $w_3$
and $w_6$ the immediate workload at Classes
3 and 6, that is, the sum of the residual service times of all of 
the jobs currently at these classes.  (Only 
the immediate workloads at these classes is used.)

The random analogs of the quantities $z_k$, $z$, 
$w_3$, $w_6$, $u_1$, and $u_4$ corresponding 
to $X(t)$ will be denoted by $Z_k(t)$, $Z(t)$,
$W_3(t)$, $W_6(t)$, $U_1(t)$, and $U_4(t)$.  
We will also employ
$A_k(t)$ to denote the total number of arrivals to Class k
over times $(0,t]$, 
and by $D_k(t)$ the total number of departures from Class k 
over this time interval.

 \section{The induction step}
\label{sectioninductionstep}

In this section, we state the induction
step, Theorem \ref{propinductionstep}, and then prove
Theorem \ref{theorem1} assuming  Theorem \ref{propinductionstep}.  The theorem 
asserts that, at the random time $T$, the number of jobs at Class 5 is a large multiple of the
number originally at Class 2 and there are few jobs or work elsewhere, 
if $\delta$ is small (and hence $M$ is large).
\begin{theorem}
\label{propinductionstep}
Let $X(t)$ be the Markov process associated with the queueing network
in Figure \ref{fig1} satisfying (\ref{entrancelaws})-(\ref{means}).
Suppose that, for large $M$,
$\delta = 1/ M^{1/15}$, and $N\ge 2M/\delta$,
\begin{equation}
\label{eqtime0-1}
Z_2 (0) = N, \quad W_3 (0) \le \delta^2 N, \quad 
\sum_{k\neq 2,3} Z_k(0) \le \delta N.
\end{equation}
Then there exists a 
stopping time $T$, satisfying 
$T \in [N/ 3\delta, 3N/\delta]$, such that
\begin{align}
\label{eqinductionstepA}
Z_5(T) \ge N/4\delta, \quad &  Z_1(T) + Z_2(T) \le 10^3 \delta N, \\
\quad  Z_3(T) = Z_4(T) & = 0, 
\quad W_6 (T) \le \delta^3 N,  \notag
\end{align}
and
\begin{equation}
\label{eqinductionstepB}
Z (t) \ge N/4,  \quad \text{ for all } t\in [0,T],
\end{equation}
all hold on a set $G_{N}$ with $P(G_N) \ge 1- C_{\delta}\text{e}^{-c_{\delta}N}$
for some $C_{\delta},c_{\delta} >0$ depending on $\delta$.
\end{theorem}

\begin{proof}[\!\!Proof of Theorem \ref{theorem1} assuming Theorem  \ref{propinductionstep}]

Suppose $X(0)$ satisfies the assumptions of Theorem \ref{propinductionstep}
for a given $N$.
Since the evolution of jobs along the upper and lower routes
of the network in Figure \ref{fig1} is symmetric, one can repeatedly 
iterate the theorem by switching the roles of 
Classes 1-3 with those of Classes 4-6, and applying the strong Markov property.
One obtains in this manner a sequence of stopping times $T_0, T_1, T_2, \ldots$,  
with $T_0 = 0$ and
$T_n \in [(N/3\delta)^n, (3N/\delta)^n]$ for $n\ge 1$, such that 
\begin{equation}
\label{eqinductionstepC}
Z (t) \ge (1/4\delta)^n \delta N,  
\quad \text{ for all } t\in [T_{n-1},T_n] \, \text{ and }\,  n\in \mathbb{Z}_+,
\end{equation}
holds on a set $G_N^{\infty}$, with 
\begin{equation}
\label{eqGlimit}
P(G_N^{\infty}) \ge 1- \sum_{n=1}^{\infty} 
C_{\delta} \text{e}^{-c_{\delta}(1/4\delta)^{n}\delta N}.
\end{equation}
The right hand side of (\ref{eqGlimit}) can be made arbitrarily close to 1 by choosing $N$ sufficiently large.  
On $G_N^{\infty}$, one has $\liminf_{t\rightarrow\infty} Z(t)/t > 0$.
So, Theorem \ref{theorem1} will follow by showing, on the set where
$\liminf_{t\rightarrow\infty} Z(t) < \infty$, that
$X(t)$ satisfying 
(\ref{eqtime0-1}) must eventually occur      
 for some $N \ge N_0$ and arbitrarily large $N_0$.

Since $\nu$ has a positive density on $(\gamma M, 2M)$
with $M$  large,
one can check that 
 $B:= \{z = 0 \text{ and } u_1 \le 1/M^2\}$ is accessible, with uniform probability 
by a fixed time, from any state $x \in \mathscr{S}$ with $z \le z_0$ and given $z_0$.
So, off of the set where $\lim_{t\rightarrow \infty} Z(t) = \infty$, $B$ will be revisited 
at arbitrarily large times.

Set $t_1 = 2(2\delta^3 + 1/M^2)N_0$, where $N_0$ is large. We claim that, for $X(0) \in B$,
\begin{equation}
\label{returnto(10)}
P\big(X(t_1) 
\text{ satisfies 
(\ref{eqtime0-1})}, \text{ for some } N\ge N_0 \big) 
\ge \text{e}^{-4N_0/M} \!/  2.
\end{equation}
It follows from (\ref{returnto(10)}) and the previous paragraph that,
off of the set where $\lim_{t\rightarrow \infty} Z(t) = \infty$,
$X(t)$ will eventually satisfy
(\ref{eqtime0-1}) for $N \ge N_0$ and arbitrarily large $N_0$.  This will complete the proof of the theorem upon demonstration of
 (\ref{returnto(10)}).

To demonstrate  (\ref{returnto(10)}), restart $X(t)$ at $x \in B$.
Set $N' = 2N_0$, and denote by $F_{N'}$ the event on which  (a) at Class 1, 
the first  $N'$ non-residual interarrival times are each
$1/M^2$ (``short")  and the next $\lceil 2\delta^3 N'/M \rceil $ interarrival times are all at least
$M$ (``long") and (b) at Class 4, the first $\lceil(2 \delta^3 + 1/M^2)N'/M \rceil$ 
non-residual interarrival times are all
long.  On $F_{N'}$,  only a few jobs enter the network over 
$(0,t_1]$, other than the $N'$ jobs corresponding to short interarrivals.
Since short interarrivals occur with probability $1 - 1/M$, long interarrivals
with probability $1/M$, and $\delta = 1/M^{1/15}$, one can check that
\begin{equation}
\label{boundonF}
 P(F_{N'}) \ge \text{e}^{-2N'/M} = \text{e}^{-4N_0 /M}.
\end{equation}

The service time of jobs at Class 1 is deterministic with
$m_1 = \delta^3 >> 1/M$.  On $F_{N'}$,  there will therefore be almost 
$N'$ jobs at Class 1 at time $N'/M^2$, and few jobs elsewhere in the network.  
The service time of jobs at Class 2 is memoryless with $m_2 = 1 - \delta$. 
After a further 
elapsed time of $2\delta^3 N'$, there will therefore  be, with high
probability, few jobs anywhere
in the network except at Class 2, where there will be $N$ jobs with 
$N \approx N' = 2N_0$.
At Class 3, the immediate workload will be much less than
$\delta^2 N'$.  So, $X(t_1)$ will satisfy the assumptions of 
(\ref{eqtime0-1}) for some $N$, with $N \ge N_0$. 
Together with (\ref{boundonF}), this implies
(\ref{returnto(10)}), which completes the proof of the theorem.

\end{proof}

\section{Basic ideas behind the proof of Theorem \ref{propinductionstep}} 
\label{secintuition}

The proof of Theorem \ref{propinductionstep} employs
reasoning similar in spirit to that used in \cite{LuKu91} and \cite{RySt92}, where the number of jobs in the
network increases proportionately over periodic ``cycles", during which the
uneven distribution of jobs ``starves" stations for work.  The reasoning
is trickier for the LIFO discipline, both because of the basic nature of the discipline, and
because of the possibility that many partially served jobs will accumulate at some
of the classes.  Here, we motivate the network topology in Figure \ref{fig1}, and 
the choice of arrival and service times in (\ref{entrancelaws}) and (\ref{means})
used to demonstrate Theorem \ref{propinductionstep}. 

The main reason for the choice of uneven interarrival times and short service times
 at Class 4 is to in effect create a low priority class there.  
Because of the large gaps in arrivals caused by the rare interarrival times of length at least
$\gamma M$ and the much more common extremely short interarrival times of length
$1/M^2$, overwhelmingly most arrivals are tightly bunched together, with large gaps in
between.  Together with the assumptions on the other classes, this will ensure that most
of the arrivals at Class 3 occur during  these gaps and so, because of the LIFO
discipline, receive a higher priority
of service than do the bunched together arrivals at Class 4.

In order for this picture to hold, one needs the flow of jobs 
to Class 3 to be evenly
spaced. This is 
accomplished by the consistent service of jobs at Class 2:
the service rule there is exponentially distributed and hence memoriless, which
ensures an even flow of jobs to Class 3 as long as Class 2 is not empty.  Moreover,
since the mean service time at Class 3 is only slightly greater than it is at Class 2, work
can only accumulate slowly at Class 3.  (Many mostly served jobs could conceivably
accumulate there.)

The mean service time $m_2 = 1 - \delta$ at Class 2 is only slightly less than the mean time
for jobs to arrive at Class 1. 
So, as long as jobs arriving at Class 1 proceed quickly to Class 2,
 approximately $N/\delta$ jobs will need to be served at Class 2 before Class 2 is first empty,
if $Z_2(0) =N$.  So, the time $S_1$ at which
Class 2 first empties will be approximately $N/\delta$.

By time $S_1$, approximately $N/\delta$ jobs will have entered the network at Class 4.
The jobs
at Class 3 will have
priority over most of the jobs arriving at Class 4; since Class 3 will 
be empty only a small fraction of the time before Class 2 is empty,  few jobs arriving
at Class 4 over $(0,S_1]$ will have completed service by time $S_1$.

As reasoned above, there is comparatively little work remaining at Class 3
at time $S_1$.  There are also
comparatively few jobs at any of the other classes, aside from Class 4:   Since few jobs are
served at Class 4 up until time $S_1$, Class 5 will experience a minimal load over that time
and so will have few jobs at time $S_1$.
For the same reason, few jobs will arrive at Class 6 from Class 5.  
Since $m_1 = \delta^3 << 1$,
jobs in Class 1 require little service.  Consequently, there will
be few jobs at the station that comprises Classes 1 and 6 at time $S_1$.  

To sum up:
At time $S_1$, Class 4 has approximately $N/\delta$ jobs, Class 3 has comparatively little
work remaining, and all other classes have comparatively few jobs.    Moreover, over $[0,S_1]$, there will always be 
at least on
the order of $N$ jobs at either Class 2 or Class 4, and so at least this many jobs in the 
network.  These conclusions are stated in
Theorem \ref{theorembehavioratS}, which summarizes the behavior of the queueing network
up until time $S_1$.

We designate by $T = S_1 + S_2$ the time after $S_1$ at which the station comprising Classes 3
and 4 is first empty.  Because of the relatively little work at time $S_1$ remaining at Class 3, the short
service time $\delta^3$ at Class 4, and the small number of jobs
arriving from elsewhere over $(S_1,T]$, $S_2$ will be of order 
  $\delta^2 N$.  Because of the
relatively short timespan $[S_1,T]$, nearly all of the order of $N/\delta$ jobs arriving 
at Class 5 from Class 4 will still be at Class 5 at time $T$.  This gives the desired lower bound on 
$Z_5 (T)$ in (\ref{eqinductionstepA}) of Theorem \ref{propinductionstep} and
the lower bound in (\ref{eqinductionstepB}) on $Z(t)$, for $t\in (S_1,T]$.
Relatively few jobs will have arrived at Class 1 over $(S_1, T]$,
and so $Z_1(T) +Z_2(T)$ will be sufficiently small for (\ref{eqinductionstepA}).  
By the definition of $S_2$, $Z_3(T) =Z_4(T) = 0$.

The number of jobs 
arriving at Class 6, which is of order of magnitude $\delta^2 N$, will nevertheless be too great for the recursion argument
we wish to employ.  However, $m_6 - m_5 = \delta^3$ is sufficiently small so that the
amount of work $W_6(T)$ that accumulates at Class 6 over the
timespan $\delta^2N$ of $(S_1,T]$ is less than 
$\delta^3 N$, which is
the desired bound on $W_6(T)$ in (\ref{eqinductionstepA}).   With this last bound,
we have thus motivated all of the bounds in (\ref{eqinductionstepA}) and
(\ref{eqinductionstepB}).   This completes our
motivation behind the proof of Theorem \ref{propinductionstep}.

We conclude this section with some
remarks on the choices of service times we have made in (\ref{means})
and on the stability of the LIFO discipline for single class networks.

{\bf Remark 1} $\,$ Classes 2 and 5 are stipulated to have exponentially distributed service times.
The proofs of Theorems \ref{theorem1} and \ref{propinductionstep}  would be
essentially the same if we replaced the deterministic
times of Classes 1, 3, 4, and 6 by exponentially
distributed service times.  However, Corollary \ref{corollary1} would then not follow from 
Theorem \ref{theorem1} since the distributions resulting by adding the service times at
Classes 3.1,...,3.L and 6.1,...,6.L
 would be gamma and not exponentially distributed. 
 
{\bf Remark 2} $\,$ In this paper, we 
consider preemptive LIFO, rather than nonpreemptive
LIFO, where a job currently in service at a class 
completes its service before more recent arrivals
are served.  For nonpreemptive LIFO, Theorem
\ref{theorem1} and its corollary continue to hold under the same
assumptions.  In that setting, one has the option
of replacing the exponentially distributed service
times at Classes 2 and 5 with deterministic times
having the same means, since there can be a most one
partially served job at each of these classes at any
given time and therefore no sudden arrival of
many jobs at Classes 3 and 6.  Similarly, one no longer needs to employ
the immediate workload rather than the number of jobs
for bounds at Classes 3 and 6.

{\bf Remark 3} $\,$Theorem \ref{theorem1} demonstrates the instability
of a family of subcritical multiclass queueing networks with the 
preemptive LIFO discipline.  As stated in Remark 2, an analogous result
 holds for the nonpreemptive LIFO discipline.  
Are subcritical single class queueing networks with either of
these disciplines necessarily stable?  For any nonpreemptive discipline of a single class queueing network, it
is easy to see that the
order of service of jobs at a station does not affect the stability of the queueing network, provided knowledge of their service times is not used.
The FIFO discipline is stable for subcritical single class
queueing networks, assuming the external
arrivals satisfy (\ref{unboundedness}) and
(\ref{abscontinuity}) (see, e.g., \cite{Br08} for
references).  Consequently, so are subcritical single class nonpreemptive LIFO queueing networks.  Conditions under which subcritical single class queueing networks 
with the preemptive LIFO discipline are stable or unstable appear not to be known.

\section{Behavior up until time $S_1$}
\label{sectionuptoS}

We demonstrate Theorem \ref{propinductionstep} by dividing the
time interval $[0,T]$ into two subintervals, $[0,S_1]$ and $[S_1, T]$, where
$S_1$ is the first time at which Class 2 is empty.
The length of $[0,S_1]$ will be of order $N/\delta$, whereas $[S_1,T]$ will be
comparatively short.  Most of the effort in showing
Theorem \ref{propinductionstep} will be in
analyzing the behavior of the queueing network over $[0,S_1]$, which will
be done in this section.  The main result is 
Theorem \ref{theorembehavioratS}, which gives bounds on $Z_k(S_1)$, for
$k\neq 3$, and $W_3(S_1)$, as well as a lower bound on $Z(t)$ over $[0,S_1]$.

\begin{theorem}
\label{theorembehavioratS}
Let $X(t)$ be the Markov process associated with the queueing network
in Figure \ref{fig1} satisfying (\ref{entrancelaws})-(\ref{means}), with
initial conditions satisfying those of 
Theorem \ref{propinductionstep} for some $N\ge 2M/\delta$. 
Then there exists a 
stopping time $S_1$, with 
$S_1  \in [N/ 2\delta, 2N/\delta]$, such that
\begin{align}
\label{behavioratS_1}
& Z_4(S_1) \in [N/3\delta, 3N/\delta], \quad Z_2(S_1) = 0, \\
W_3(S_1) & \le 7\delta^2 N,
\quad Z_k(S_1) \le 7\delta^2 N \quad  \text{ for } k =  1,5,6, \notag
\end{align}
and
\begin{equation}
\label{behaviorattimesbeforeS}
Z (t) \ge N/3,  \quad \text{ for all } t\in [0,S_1],
\end{equation}
all hold on a set $G_{S_1}$ with $P(G_{S_1}) \ge 1- C_{\delta}\text{e}^{-c_{\delta}N}$
for some $C_{\delta},c_{\delta}>0$.
\end{theorem}

\subsection{Two large deviations lemmas}
\label{seclemmas}

In this subsection, we state two basic large deviation lemmas. 
The first, Lemma \ref{lemmadeviations}, 
 can be obtained by using the moment generating function and Markov's inequality
(see, e.g., Theorem 15 and Lemma 5,
in Chapter 3 of \cite{Pe75}).  
\begin{lemma}
\label{lemmadeviations}
Let $X_1, X_2,\ldots$ be i.i.d. positive random variables with mean $\mu$ and  
$P(X_1  \ge x) \le  \text{e}^{-\alpha x}$ for $x\ge x_0$ and some $\alpha,x_0 >0$.  Set $S_n = \sum_{i=1}^n X_i$ and $\beta^{\star} = \beta (\beta \wedge 1)$.
Then there exists $c>0$ such that, for all $\beta > 0$,
  \begin{equation}
\label{eqLD}
P(|S_n - \mu n |/n \ge \beta) \le \text{e}^{-c \beta^{\star}n} \qquad \text{ for all } n\ge 0.
\end{equation}

\end{lemma}
Setting $N_t = \max \{n: S_n \le t \}$, 
 one obtains by inverting (\ref{eqLD}) and a bit of calculation: 
\begin{equation}
\label{invertedLD}
P(|N_t - \mu^{-1}t |/t \ge \beta) 
\le C\text{e}^{-c \beta^{\star} t} 
\qquad \text{ for all } t\ge 0,
\end{equation}
for appropriate $C,c>0$
not depending on $\beta$.
This  bound will be applied repeatedly in the section to $D_2(t)$ and $A_4(t)$,
as well as to other departure and arrival times. 
 Summing (\ref{invertedLD}) over $t = t_0, t_0 +1, \ldots$ and interpolating
in between, one  obtains that, for appropriate $C>0$ and each $t_0 \ge 0$,
\begin{equation}
\label{LDforallt}
 P(|N_t - \mu^{-1}t |/t \ge \beta \,\, \text{ for any } t\ge t_0 ) 
\le C (\beta^{\star}\wedge 1)^{-1} \text{e}^{-c \beta^{\star} t_0}. 
\end{equation}
(The explicit dependence on $\beta$ in the bounds (\ref{eqLD})-(\ref{LDforallt}) will only be used in Lemma \ref{lemmadrift} and Proposition \ref{propclass4control}. Bounds in other
applications will be of the form 
$C_{\delta} \text{e}^{-c_{\delta}t}$, where the relationship with $\beta$ is 
suppressed.)

%
%

For the next lemma, 
consider a queue at which jobs arrive according to a rate-($1+ \eta$)
Poisson process and each job requires 1 unit of time to be served
before exiting the queue.  Let $W_t$ denote the immediate workload at time $t$,
that is, the amount of time required for all jobs then at the queue to
be served, provided no other jobs arrive.  The lemma will be used
in Propositions \ref{propclass4control} and \ref{propclass34control} to
obtain bounds
on the state at Classes 3 and 4, until Class 2 is
empty.

\begin{lemma}
\label{lemmadrift}
Define $W_t$ as above, with $W_0 = 0$, and set $B = \{t\ge 0 : W_t = 0 \}$.  
 Then, for $\eta \in (0,1]$ and appropriate $C,c > 0$,
\begin{equation}
\label{driftbound}
P(|B| \ge x) \le (C/\eta^2) \text{e}^{-c\eta^2 x} \quad  \text{for all } x\ge 0, 
\end{equation}
and
\begin{equation}
\label{LDforWestimate}
P(W_t \ge 2\eta t \text{ for some } t\ge t_0) \le (C/\eta^2) \text{e}^{-c\eta^3 t_0}
\quad \text{for all } t_0  \ge 0.
\end{equation}
\end{lemma}
\begin{proof}
To obtain  (\ref{driftbound}) and (\ref{LDforWestimate}), we compare the
process $W_t$ with  
$X_t = Y_t - t$, where $Y_t$ is a rate-($1+ \eta$)
Poisson process and
$X_0 = 0$.  Set $B^X = \{t\ge 0 : X_t \le 0 \}$.   Coupling the processes
$W_t$ and $X_t$ by allowing the same random input for each, 
$W_t \ge X_t$ for all $t$.  Therefore, (\ref{driftbound}) will follow by
showing its analog,
\begin{equation}
\label{driftboundsimple}
P(|B^
X| \ge x) \le (C/\eta^2)\text{e}^{-c\eta^2 x}   \quad  \text{for } x\ge 0,
\end{equation}
for appropriate $C,c > 0$.  Setting  
$B_t = \{s\in [0,t]: W_s = 0 \}$, one has $W_t - X_t = |B_t| \le |B|$.  
Plugging this into (\ref{driftbound}), one can check that
 (\ref{LDforWestimate}) will follow from
\begin{equation}
\label{LDforWestimatesimple}
P(X_t \ge 2\eta t \text{ for some } t\ge t_0) \le (C/\eta^2)\text{e}^{-c\eta^2 t_0} ,
\end{equation}
for appropriate $C,c > 0$.

The bound in (\ref{LDforWestimatesimple}) follows directly from
(\ref{LDforallt}).
To show (\ref{driftboundsimple}), note that, by integrating (\ref{invertedLD}), 
\begin{equation}
\label{LDuupperX_t}
\int_{t_0}^{\infty}P(X_t \le 0)dt \le (C/c\eta^2) \text{e}^{-c\eta^2 t_0},
\end{equation}
with $C,c>0$ not depending on $\eta$ for $\eta \in (0,1]$.  
On $|B^X| \ge x$, $X_t \le 0$ must occur for at least 1 unit of time on $t\ge x-1$,
so 
\begin{equation*}
P(|B^X| \ge x) \le \int_{x-1}^{\infty}P(X_t \le 0)dt \le 
(C'/c\eta^2) \text{e}^{-c\eta^2 x},
\end{equation*}
for appropriate $C'>0$, which implies (\ref{driftboundsimple}) for a new choice of $C$.
 
\end{proof}

\subsection{Demonstration of Theorem \ref{theorembehavioratS}}  In order to demonstrate
Theorem \ref{theorembehavioratS}, we employ six propositions on the evolution
of the queueing network that is due to the continuing service of jobs at Class 2 over
$[0,S_1]$.  An outline of the reasoning is given in Section \ref{secintuition}.

For Proposition \ref{propclass4control}, we will employ Lemmas
\ref{lemmaclusters} and \ref{lemmaXandtildeX}.
We begin by observing that the interarrival 
times at Class 4 
are either very long or very short. 
This allows us to
decompose the sequence of arrivals into \emph{clusters}, with an individual cluster
$\mathscr{C}$ consisting of
a finite sequence of jobs where the interarrival times between members of
the cluster
are each of length $1/M^2$, and these jobs are preceded and followed by 
interarrival times of length $t\ge \gamma M$. 
(The first arrival after time $0$ is assumed to begin a cluster.)
We denote successive clusters by
$\mathscr{C}_i$, $i\in \mathbb{Z}_+$ (see Figure \ref{fig3}).  Since $\gamma \sim 1$, 
the following lemma is immediate.

\begin{figure}
\centering
\begin{tikzpicture}[thick,scale=0.90, every node/.style={scale=0.85}]
\draw[thick] (-4.6,0) -- (7.7,0); 
\draw[thick] (-4.6,-.3) -- (-4.6,.3); 
\draw[thick] (6.9,-.3) -- (6.9,.3); 

\filldraw (-5.15,-.6)  node [anchor=west] {$t=0$};
\filldraw (6.35,-.6)  node [anchor=west] {$t=t_0$};
\filldraw (0.66,-1.0)  node [anchor=west] {4 clusters};

\filldraw (-3.10,1.2)  node [anchor=west] {$= 1/M^2$};
\filldraw (1.15,1.2)  node [anchor=west] {$\sim M$};
\filldraw (4.40,1.2)  node [anchor=west] {$= 1/M^2$};

\draw[thick, ->] (0.55,-1) -- (-3.6,-.30); 
\draw[thick, ->] (0.55,-.8) -- (-0.4, -.30); 
\draw[thick, ->] (2.5,-.8) -- (3.2,.-.30); 
\draw[thick, ->] (2.5,-1) -- (6.20,-.30); 

\draw[thick, ->] (-2.75,.8) -- (-4.35,.15); 
\draw[thick, ->] (-2.45,.8) -- (-3.59,.15); 
\draw[thick, ->] (-2.10,.8) -- (-1.35,.15); 
\draw[thick, ->] (-1.80,.8) -- (-.31,.15); 

\draw[thick, ->] (4.75,.8) -- (2.91,.15); 
\draw[thick, ->] (5.05,.8) -- (3.91,.15); 
\draw[thick, ->] (5.40,.8) -- (6.21,.15); 
\draw[thick, ->] (5.65,.8) -- (6.69,.15); 

\draw [decorate,decoration={brace,amplitude=7pt},xshift=-4pt,yshift=0pt]
(0.85,0.5) -- (2.6, 0.5) node [black,midway,xshift= 0.6cm]
{\footnotesize };

\fill (-4.45,0) circle (0.08);
\fill (-4.20,0) circle (0.08);
\fill (-3.95,0) circle(0.08);
\fill (-3.70,0) circle (0.08);
\fill (-3.45,0) circle(0.08);

\fill (-1.45,0) circle (0.08);
\fill (-1.20,0) circle (0.08);
\fill (-.95,0) circle(0.08);
\fill (-.70,0) circle (0.08);
\fill (-.45,0) circle(0.08);
\fill (-.20,0) circle (0.08);
\fill (.05,0) circle(0.08);
\fill (.30,0) circle (0.08);
\fill (.55,0) circle(0.08);

\fill (2.55,0) circle (0.08);
\fill (2.80,0) circle (0.08);
\fill (3.05,0) circle (0.08);
\fill (3.30,0) circle (0.08);
\fill (3.55,0) circle(0.08);
\fill (3.80,0) circle (0.08);
\fill (4.05,0) circle(0.08);

\fill (6.05,0) circle (0.08);
\fill (6.30,0) circle (0.08);
\fill (6.55,0) circle (0.08);
\fill (6.80,0) circle (0.08);
\fill (7.05,0) circle(0.08);
\fill (7.30,0) circle (0.08);

\end{tikzpicture}
\caption{A realization depicting the arrival times of jobs at either Class 1
or at Class 4, with dots indicating these arrival times.
In this case, there are four clusters $\mathscr{C}_i$
overlapping $[0,t_0]$.  The distance between jobs within a cluster is
always $1/M^2$; the distance between clusters is random and approximately
$M$ (since  $\gamma \sim 1$).
}
\label{fig3}
\end{figure}
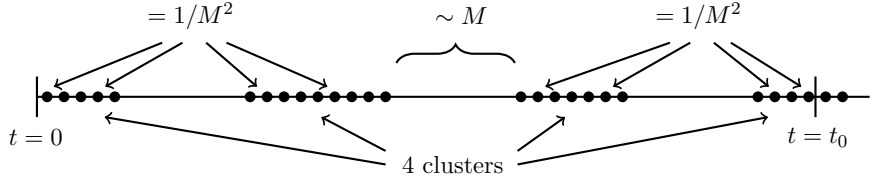

\begin{lemma}
\label{lemmaclusters}
The number of clusters at Class 4 overlapping the time interval $(0,t_0]$ is at most $\lceil 2t_0/M \rceil$ for any $t_0 >0$.
\end{lemma}

We denote by $\mathscr{L}_i$ the number of jobs in
$\mathscr{C}_i$ and by $U_i$ the time of arrival of the first job of 
this cluster.  
Also denote by $Y_i$ the amount of time on the interval
$(U_i + \mathscr{L}_i/M^2, U_{i+1}] \cap (0,t_0]$ 
that Class 3 does not have any jobs that arrived 
after $U_i + \mathscr{L}_i/M^2$.  Over
$(U_i + \mathscr{L}_i/M^2, U_{i+1}]$, 
all Class 3 jobs arriving after $U_i + \mathscr{L}_i/M^2$
will have priority over the jobs in Class 4 because of the LIFO rule.

We also introduce $\tilde{Y}_i$,
which is defined in the same way as $Y_i$, but for a \emph{modified process}  
where the arrival stream of jobs at
Class 3 is now given by a rate-$(1-\delta)^{-1}$ Poisson process instead of by actual
departures from Class 2, and the time interval 
$(U_i + \mathscr{L}_i/M^2, U_{i+1}]$ corresponding to 
$\tilde{Y}_i$  is not intersected by $(0,t_0]$.  
One can couple the original and modified processes so
that an arrival at Class 3 for the modified process always
occurs whenever an arrival for the original process occurs. 
The sequence  $(Y_i)_{i\in \mathbb{Z}_+}$ is not i.i.d.  However, setting
\begin{equation}
\label{defofA}
A_{t_0} = \{Z_2(t) > 0 \, \text{ for all } t\in [0,t_0) \},
\end{equation}
it is easy to
check the following:
\begin{lemma}
\label{lemmaXandtildeX}
The sequence $(\tilde{Y}_i)_{i\in \mathbb{Z}_+}$ is i.i.d.
For $\omega \in A_{t_0}$ and $U_{i+1} \le t_0$, $Y_i = \tilde{Y}_i$; for 
$\omega \in A_{t_0}$ and all $i$,
$Y_i \le \tilde{Y}_i$.
\end{lemma}

Proposition \ref{propclass4control} provides an upper bound on the number of jobs
that can depart from Class 4 on the event $A_{t_0}$.  The bound is due
to ``most"  arriving jobs at Class 3 having higher priority than ``most" jobs at Class 4, and Class 3 seldom being empty.  This reasoning will employ the relatively low number of clusters of
jobs at Class 4 together with there being ``few" jobs in each cluster that can have higher priority than the
jobs in Class 3.  In the proposition, the choice of the term $\delta^3$ 
is somewhat arbitrary -- we require at least this power, but
$\delta^n$ could instead be used if we defined $\delta = M^{1/(n+12)}$.

We denote by $D_4^o(t)$ 
the number of jobs originally in Class 4 that have departed
the class by time $t$

\begin{proposition}
\label{propclass4control}
Let $X(t)$ be the Markov process associated with the queueing network
in Figure \ref{fig1} satisfying (\ref{entrancelaws})-(\ref{means}), with
any initial condition. 
Then, for appropriate 
$C_{\delta},c_{\delta} >0$,
\begin{equation}
\label{class4control}
P(D_4(t_0) \ge \delta^3 t_0 + D_4^o (t_0) \, ; A_{t_0}) \le C_{\delta}\text{e}^{-c_{\delta} t_0},
\end{equation}
where $A_{t_0}$ is as in (\ref{defofA}).
\end{proposition}
\begin{proof}
By time 
$U_i + \mathscr{L}_i/M^2$,
all jobs in $\mathscr{C}_i$ have arrived at Class 4, 
and jobs arriving at Class 3 after then 
have priority over these jobs.
Consequently, the time over $(U_i, U_{i+1}]$ that is 
available for service of  jobs in Class 4 is at most
$\mathscr{L}_i/M^2 +Y_i$.  Because of Lemma \ref{lemmaXandtildeX}, 
the dominating sequence $(\mathscr{L}_i/M^2 +\tilde{Y}_i)_{i\in \mathbb{Z}_+}$ is i.i.d.

Each job at Class 4 requires $\delta^3$ amount of service.  
Applying Lemmas \ref{lemmaclusters} and \ref{lemmaXandtildeX}, it follows that,
on  $A_{t_0}$,
\begin{equation}
\label{L+Ybounds}
D_4(t_0) - D_4^o(t_0) \le 
\sum_{i=1}^{\lceil 2t_0/M\rceil} (\mathscr{L}_i/M^2 +Y_i)/\delta^3
\le \sum_{i=1}^{\lceil 2t_0/M\rceil} (\mathscr{L}_i/M^2 +\tilde{Y}_i)/\delta^3.
\end{equation}
It follows quickly from (\ref{entrancelaws}) that
\begin{equation*}
\label{clusterlength}
P(\mathscr{L}_i \ge \ell) \le \text{e}^{-\ell /M} 
\quad \text{for any } \ell \ge 0.
\end{equation*}
On the other hand, by (\ref{driftbound}) of Lemma \ref{lemmadrift}, with $\eta = \delta^3/2$, 
\begin{equation*}
P(\tilde{Y}_i \ge x) \le (C/\delta^6)\text{e}^{-c\delta^6x} \quad \text{for }  x\ge 0,
\end{equation*}
and some $C,c >0$. 
It follows from these two displays (the main contribution is from the latter) that
the summands on the right hand side of (\ref{L+Ybounds})
are dominated by random variables 
\begin{equation*}
\label{decompintoexpl}
V_i := (2/c\delta^9) (R_i + \text{log}[2C/\delta^6]),
\end{equation*}
where $R_i$ are independent and mean-1 exponentially distributed.

By (\ref{eqLD}), for $y\ge y_0 >0$ and some $C',c' >0$
depending on $y_0$,
\begin{equation*}
P\bigg{(}\sum_{i=1}^{b_t} R_i \ge b_t (1+y)\bigg{)} \le C'\text{e}^{-c'b_t y}.
\end{equation*}
Substituting in $V_i$ for $R_i$ and applying $\delta = M^{1/15}$, it  follows from the above two displays and a bit of computation that
\begin{align*}
\label{L+Xb}
P\bigg{(}\sum_{i=1}^{\lceil 2t_0/M \rceil} V_i \ge \delta^3 t_0 \bigg{)} &  \le 
C'\text{exp}\Big\{-c'\delta^{12} t_0 \Big(\frac{c}{2} - 2\delta^3 (\text{log}[2C/\delta^6] + 1) \Big) \Big\}  \\
&  \le C'\text{e}^{-cc'\delta^2 t_0 /4}, 
\end{align*}
with the second inequality holding since $c/2$ is the dominant term inside of the large parentheses on the
right hand side.
The desired inequality (\ref{class4control}) follows from this display and
(\ref{L+Ybounds}).

\end{proof}

We denote by $S_1$ the stopping time at which $Z_2(t)=0$ first occurs.  In the remainder of this section and in
Section \ref{proofoftheorem}, we will examine the behavior of $X(S_1)$, and then
restart the process there.  Our first result provides an elementary upper
bound on $S_1$.

\begin{proposition}
\label{propSupperbound}
Let $X(t)$ be the Markov process associated with the queueing network
in Figure \ref{fig1} satisfying (\ref{entrancelaws})-(\ref{means}), with
initial conditions satisfying $Z_1(0) + Z_2(0) \le 2N$ for some $N$.   Then
\begin{equation}
\label{Supperbound}
P(S_1 \ge 2N/\delta) \le C_{\delta} \text{e}^{-c_{\delta}N}
\end{equation}
for appropriate $C_{\delta},c_{\delta}>0$.
\end{proposition}
\begin{proof}
For any time $t$, $A_2(t) \le A_1(t) + Z_1(0)$; adding $Z_2(0)$ to this
gives an upper bound on the number of jobs to have visited Class 2 by
time $t$.  On the other hand, since the interarrival distribution satisfies (\ref{entrancelaws}), with mean 1, and $m_2 = 1 - \delta$, 
it follows from (\ref{invertedLD}) that
\begin{align*}
P\big(D_2(2N/\delta) \le  A_1(2N/\delta) + Z_1(0) + Z_2(0)\,\,;\,\, & Z_2(s) > 0 \, \text{ for all } s\in [0,t]\big) \\
& \quad \le  C_{\delta} \text{e}^{-c_{\delta}N},
\end{align*}
for appropriate $C_{\delta},c_{\delta}>0$.   Off of the exceptional set in the display,
$S_1 < 2N/\delta$, which implies (\ref{Supperbound}).
 
\end{proof}
The following result is a quick consequence of Propositions \ref{propclass4control} and \ref{propSupperbound}. 
 It provides an upper bound on the number of jobs
ever to visit Classes 5 and 6 over $t\in [0,S_1]$ and is important for
establishing the long-term cyclical growth of $Z(t)$.
\begin{corollary}
\label{corclass56control}
Let $X(t)$ be the Markov process associated with the queueing network
in Figure \ref{fig1} satisfying (\ref{entrancelaws})-(\ref{means}), with
initial conditions satisfying $Z_1(0) + Z_2(0) \le 2N$ for some $N$.   Then
\begin{equation}
\label{class4exitcontrol}
P\big(D_4(S_1) \ge 2\delta^2 N \!+\! Z_4(0)\big) \,\le \, 
P\big(D_4(S_1) \ge  2\delta^2 N \!+\! D_4^o (S_1)\big) \,\le \,
C\,_{\delta} \text{e}^{-c_{\delta}N}
\end{equation}
for appropriate $C_{\delta},c_{\delta}>0$.  Hence, denoting by $\mathscr{V}_{S_1}$ the total
number of jobs ever to be in either Class 5 or Class 6 over $[0,S_1]$,
\begin{equation}
\label{class56control}
P\Big(\mathscr{V}_{S_1} \ge 2\delta^2 N + \sum\nolimits_{k=4}^6 Z_k(0)\Big)  
\le C_{\delta} \text{e}^{-c_{\delta}N}.
\end{equation}
\end{corollary}

We employ Corollary \ref{corclass56control} to obtain the following upper
bound on $Z_1(t)$, for $t\in [0,S_1]$, and 
the following lower bound on $S_1$.  

\begin{proposition}
\label{propclass12control}
Let $X(t)$ be the Markov process associated with the queueing network
in Figure \ref{fig1} satisfying (\ref{entrancelaws})-(\ref{means}), with
initial conditions satisfying those of 
Theorem \ref{propinductionstep} for some $N \ge 2M/\delta$. 
Then 
\begin{equation}
\label{class1control}
P(Z_1(t) \ge 7 \delta N \text{ for some } t\in [0,S_1]) \le C_{\delta} \text{e}^{-c_{\delta}N}
\end{equation}
and
\begin{equation}
\label{Slowerbound}
P(S_1 \le N/2\delta) \le C_{\delta} \text{e}^{-c_{\delta}N}
\end{equation}
for appropriate $C_{\delta},c_{\delta}>0$.
\end{proposition}
\begin{proof}  The demonstration of 
 (\ref{class1control}) is rather long; by employing
 (\ref{class1control}), the demonstration of 
(\ref{Slowerbound}) will be quick.

\emph{Demonstration of (\ref{class1control})}.
For given $s$, set
\begin{align*}
&\mathscr{A}_s = \{Z_1(t) = 0 \text{ for some } t\in [s,s+3\delta N]\}, \\
\mathscr{B}_s  = & \{Z_1(t) \ge 7\delta N \text{ for some } 
t\in [s+3\delta N, s+6\delta N]\}.
\end{align*}
We claim that
\begin{equation}
\label{class1bothbounds}
P(\mathscr{A}_s \cap \mathscr{B}_s) \le C_{\delta} \text{e}^{-c_{\delta} N} \quad \text{and} \quad  
P(\mathscr{A}^c_s) \le C_{\delta} \text{e}^{-c_{\delta} N} 
\end{equation}
for $s\le S_1 - 3\delta N$ and appropriate $C_{\delta},c_{\delta}>0$.  On $\mathscr{A}_s$, there is at most time
$6\delta N$ for at least $7 \delta N$ jobs to arrive at empty  Class 1 in
order for
$\mathscr{B}_s$ to hold, so the first inequality in the display follows from
(\ref{entrancelaws}) and (\ref{invertedLD}) applied to
$A_1(s+ 6\delta N) - A_1(s)$. 

For the second inequality, note that, on account of (\ref{eqtime0-1}) and 
(\ref{class56control}) of 
Corollary \ref{corclass56control}, there are typically at most $2\delta N$
jobs in Class 6 to interfere with the service of Class 1 jobs over
times $[s,s+3\delta N]$, and they require only time $2\delta N$ to be 
served.  Also, the service time of Class 1 jobs is $ \delta^3$,
$Z_1(0) \le \delta N$
and, by (\ref{Supperbound}) of Proposition \ref{propSupperbound}, the
number of arrivals by time $s+3\delta N$  is at most  $A_1(2N/\delta)$.
Applyimg (\ref{invertedLD}), the time required to serve all Class 1
jobs arriving by time $S_1$ is therefore typically at most $3\delta^2 N$.
Since  $2\delta N + 3\delta^2 N < 3\delta N$,
the second inequality follows.

By (\ref{class1bothbounds}),
\begin{equation*}
P(\mathscr{B}_s) \le 2C_{\delta} \text{e}^{-c_{\delta} N}.
\end{equation*}
Denoting by $I$ the set of $s$ that are multiples of $3\delta N$ and
 $s\le S_1 - 3\delta N$, it follows that
\begin{equation}
\label{class1emptying}
P\Big(\bigcup\nolimits_{s\in I} \mathscr{B}_s \Big) \le C_{\delta} \text{e}^{-c_{\delta} N}
\end{equation}
for another choice of $C_{\delta},c_{\delta}>0$.  The bound in 
 (\ref{class1control}), but restricted to $t\in [3\delta N, S_1]$, follows
from (\ref{class1emptying}).  The extension to $t\in [0,S_1]$ follows from
(\ref{eqtime0-1}), (\ref{class56control}), 
and (\ref{invertedLD}) applied to $A_1(3\delta N)$.

\emph{Demonstration of (\ref{Slowerbound}).}  Since $N\ge 2M/\delta$, one has
$U_1 (0) \le \delta N$.  Therefore, by (\ref{entrancelaws}),
 (\ref{means}),  and 
(\ref{LDforallt}),
\begin{align}
& \,\,\,\, \, P\big(A_1(t) \le (1 - \delta^2)t - 2\delta N \text{ for any } t\ge 0\big) \le  C_{\delta} \text{e}^{-c_{\delta}N},
\label{noslow1arrivals}\\
&P\big(D_2(t) \ge (1 + \delta +2\delta^2)t + \delta N \text{ for any } t\ge 0 \big) \le  C_{\delta} \text{e}^{-c_{\delta} N}, \label{nofast2departures}
\end{align}
for appropriate $C_{\delta},c_{\delta}>0$.  Together with $Z_2(0)\ge N$ and (\ref{class1control}), 
(\ref{noslow1arrivals}) and (\ref{nofast2departures})
imply that
$ Z_2(t) >0$ {for } $t\le N/2\delta$,
off of the exceptional sets in (\ref{class1control}), (\ref{noslow1arrivals}), and (\ref{nofast2departures}).  This implies (\ref{Slowerbound}).
 
\end{proof}

We employ the preceding three propositions to obtain the following upper
bound on $W_3(S_1)$ and lower bound on $Z_4(t)$, for $t\in [N/3,S_1]$. 
We require a bound
on $W_3(S_1)$ rather than on $Z_3(S_1)$ because of the possible presence
of many mostly served jobs at Class 3.

\begin{proposition}
\label{propclass34control}
Let $X(t)$ be the Markov process associated with the queueing network
in Figure \ref{fig1} satisfying (\ref{entrancelaws})-(\ref{means}), with
initial conditions satisfying those of 
Theorem \ref{propinductionstep} for some $N\ge 2M/\delta$.
Then
\begin{equation}
\label{class3control}
P\big(W_3(S_1) \ge 7\delta^2 N\big) \le C_{\delta} \text{e}^{-c_{\delta}N}
\end{equation}
and
\begin{equation}
\label{class4boundovert}
P\big(Z_4(t) \le (1-5\delta)t  \,\, \text{ for some } t\in [N/3,S_1]\big) \le C_{\delta} \text{e}^{-c_{\delta}N}
\end{equation}
for appropriate $C_{\delta},c_{\delta}>0$.
\end{proposition}
\begin{proof}
Since $N\ge 2M/\delta$, one has
$U_4 (0) \le \delta N$.
Together with (\ref{entrancelaws}) and (\ref{LDforallt}), this implies
\begin{equation*}
P\big(A_4(t) \le (1-4\delta)t \,\,\, \text{ for some } t\ge N/3 \big) \le C_{\delta} \text{e}^{-c_{\delta}N}
\end{equation*}
for appropriate $C_{\delta},c_{\delta}>0$.
Together with (\ref{class4control}) of
Proposition \ref{propclass4control}, this implies (\ref{class4boundovert}).

Since $m_2 = 1 - \delta$, $m_3-m_2 = \delta^3$, and $W_3(0)\le \delta^2 N$, 
it follows from
 (\ref{LDforWestimate}) of Lemma \ref{lemmadrift} that, for given $t_0$,
\begin{equation}
\label{helpW3lowerbound}
P\big(W_3(t) \ge 3\delta^3 t  + \delta^2 N \text{ for some } t\ge t_0 \big) \le C_{\delta} \text{e}^{-c_{\delta}t_0},
\end{equation}
for appropriate $C_{\delta},c_{\delta} > 0$. By (\ref{Supperbound}) of Proposition
 \ref{propSupperbound} and (\ref{Slowerbound}) of Proposition
 \ref{propclass12control}, $S_1 \in (N/2\delta, 2N/\delta)^c$ occurs only on an exceptional
set.  Inequality (\ref{class3control}) follows by applying these bounds 
together with that in (\ref{helpW3lowerbound}).

\end{proof}

The following corollary of Proposition \ref{propclass34control} allows us to improve on Corollary \ref{corclass56control} by bounding more precisely 
the number of jobs leaving
Class 4 after time $N/3$ and the total number of jobs to be in Classes 5 or 6 over $[N,S_1]$.

\begin{corollary}
\label{cornodepfrom4}
Let $X(t)$ be the Markov process associated with the queueing network
in Figure \ref{fig1} satisfying (\ref{entrancelaws})-(\ref{means}), with
initial conditions satisfying those of 
Theorem \ref{propinductionstep} for some $N \ge 2M/\delta$.
Then
\begin{equation}
\label{nodepfrom4}
P(D_4^o (S_1) \neq D_4^o (N/3)) \le  C_{\delta} \text{e}^{-c_{\delta}N},
\end{equation}
and
\begin{equation}
\label{class4controlwo4}
P(D_4(S_1) - D_4(N/3) \ge 2\delta^2 N) \le C_{\delta}\text{e}^{-c_{\delta} N}
\end{equation}
for appropriate $C_{\delta},c_{\delta}>0$.  Denoting by $\mathscr{V}_{S_1}^N$ 
the total
number of jobs ever to be in either Class 5 or Class 6 over $[N,S_1]$,
\begin{equation}
\label{class56controlstronger}
P(\mathscr{V}_{S_1}^N \ge 2\delta^2 N)  
\le C_{\delta} \text{e}^{-c_{\delta}N}.
\end{equation}
for a new choice of $c_{\delta}>0$.
\end{corollary}
\begin{proof}
By 
(\ref{class4boundovert}) of Proposition \ref{propclass34control},
$$ P(Z_4 (t) \le Z_4 (0) \,  \text{ for some } t \in [N/3,S_1]) \le C_{\delta} \text{e}^{-c_{\delta}N/3} $$
for appropriate $C_{\delta},c_{\delta}>0$. So, off of the exceptional set, no
job originally at Class 4 can depart over $[N/3,S_1]$ because of the LIFO property, and
(\ref{nodepfrom4}) follows.  
Inequality (\ref{class4controlwo4})
 follows from (\ref{class4exitcontrol}) of Corollary \ref{corclass56control}.

For (\ref{class56controlstronger}), note that, by (\ref{class56control}) of 
Corollary \ref{corclass56control} and the initial conditions of 
Theorem \ref{propinductionstep}, 
\begin{equation*}
\label{classes56aftertimeN}
P(\mathscr{V}_{S_1} \ge 3\delta N) \le C_{\delta} \text{e}^{-c_{\delta}N}
\end{equation*}
for appropriate $C_{\delta},c_{\delta}>0$.  It therefore follows after applying
(\ref{invertedLD}) to $A_1(N)$ and $D_5(N)$
that, for some stopping time $S_0 \in [N/3,N]$,
$$ P(Z_k(S_0) = 0 \, \text{ for } k=1,5,6) \ge 1 - C_{\delta} \text{e}^{-c_{\delta}N}$$
for a new choice of $C_{\delta},c_{\delta}>0$.  Display  (\ref{class56controlstronger})
follows from this and (\ref{class4controlwo4}).
 
\end{proof}

Employing (\ref{class56controlstronger}) of Corollary \ref{cornodepfrom4},
we obtain the following stronger version of (\ref{class1control}) of
Proposition  \ref{propclass12control}.

\begin{proposition}
\label{propimprovedclass1}
Let $X(t)$ be the Markov process associated with the queueing network
in Figure \ref{fig1} satisfying (\ref{entrancelaws})-(\ref{means}), with
initial conditions satisfying those of 
Theorem \ref{propinductionstep} for some $N \ge 2M/\delta$.
Then
\begin{equation}
\label{betterclass1control}
P\big(Z_1(t) \ge 7 \delta^2 N \, \text{ for some } t\in [N,S_1]\big) \le C_{\delta} \text{e}^{-c_{\delta}N}.
\end{equation}
\end{proposition}
\begin{proof}
The argument is the same as that for (\ref{class1control}) except for minor
changes.  For given $s$, we now set
\begin{align*}
&\mathscr{A}_s = \{Z_1(t) = 0 \text{ for some } t\in [s,s+5\delta^2 N]\}, \\
\mathscr{B}_s  = & \{Z_1(t) \ge 7\delta^2 N \text{ for some } 
t\in [s+5\delta^2 N, s+6\delta ^2N]\}.
\end{align*}
The events $\mathscr{A}_s \cap \mathscr{B}_s$ and $\mathscr{A}^c_s$
each occur with low probability:    On $\mathscr{A}_s$, there is at most time
$6\delta^2 N$ for at least $7 \delta^2 N$ jobs to arrive at empty  Class 1 in
order for $\mathscr{B}_s$ to hold.  On the other hand, because of (\ref{class56controlstronger}), 
 there are typically at most $2\delta^2 N$
jobs in Class 6 to interfere with the service of Class 1 jobs over
times $[s,s+5\delta^2 N]$.  Also, the service time of Class 1 jobs is  
$ \delta^3$ and the time required to serve all of these jobs is typically strictly less
than $3\delta^2 N$.  So the same reasoning as for (\ref{class1control}) implies that
$\mathscr{A}_s $ will typically occur.

 The remainder of the argument for 
(\ref{betterclass1control}) follows that of the proof of
Proposition  \ref{propclass12control}.

\end{proof}

The following proposition estimates $Z_4(S_1)$ and gives a lower bound
on $Z(t)$ over $[0,S_1]$; it follows quickly from previous results.

\begin{proposition}
\label{corminimalocc}
Let $X(t)$ be the Markov process associated with the queueing network
in Figure \ref{fig1} satisfying (\ref{entrancelaws})-(\ref{means}), with
initial conditions satisfying (\ref{eqtime0-1}) of 
Theorem \ref{propinductionstep} for some $N \ge 2M/\delta$. 
Then
\begin{equation}
\label{class4boundtimeS}
P\big(Z_4(S_1) \in [N/3\delta, 3N/\delta]^c \big) \le C_{\delta} \text{e}^{-c_{\delta}N}
\end{equation}
and
\begin{equation}
\label{smallZbytimeS}
P\big(Z(t) \le N/3 \,\,\, \text{ for any } t\in [0,S_1] \big) \le C_{\delta} \text{e}^{-c_{\delta}N}
\end{equation}
for appropriate $C_{\delta},c_{\delta}>0$.
\end{proposition}
\begin{proof}
The inequality in (\ref{class4boundtimeS}) follows from the upper and lower
bounds on $S_1$ given in (\ref{Supperbound}) and
(\ref{Slowerbound}), the lower bound on $Z_4(t)$ in  (\ref{class4boundovert}),
and by applying (\ref{invertedLD}) to $A_4(t)$.  

Up until time $N/2$, the inequality for
(\ref{smallZbytimeS}) follows from $Z_2(0) = N$ and by applying (\ref{invertedLD})
to $D_2(t)$; on $[N/2, S_1]$, the inequality follows from  (\ref{class4exitcontrol})
 and by again applying (\ref{invertedLD}) to $A_4(t)$.

\end{proof}

Combining the preceding results, we obtain Theorem \ref{theorembehavioratS}.
%
\begin{proof} [\! Proof of Theorem  \ref{theorembehavioratS}.]
The assertion for $S_1$ follows from (\ref{Supperbound}) of Proposition 
\ref{propSupperbound} and 
(\ref{Slowerbound}) of Proposition \ref{propclass12control}.
The assertion for $Z_4(S_1)$ follows from (\ref{class4boundtimeS}) of 
Proposition \ref{corminimalocc}, and that for $W_3(S_1)$ follows from
(\ref{class3control}) of Proposition \ref{propclass34control}.  The
assertion for
 $Z_k(S_1)$, for other $k$, follows from (\ref{class56controlstronger}) of
Corollary \ref{cornodepfrom4},
(\ref{betterclass1control}) of Proposition \ref{propimprovedclass1},
 and $Z_2(S_1) = 0$.  The assertion
for $Z(t)$ follows from 
(\ref{smallZbytimeS}) of Corollary \ref{corminimalocc}.  

\end{proof}

 \section{Demonstration of Theorem  \ref{propinductionstep}}
\label{proofoftheorem}

\begin{proposition}
\label{propatS_2}
Let $X(t)$ be the Markov process associated with the queueing network
in Figure \ref{fig1} satisfying (\ref{entrancelaws})-(\ref{means}).   Suppose
that, for some $N$ and $a \in [1,N]$,
\begin{equation}
\label{eqtimeS_2}
Z_4 (0) = N, \quad W_3(0) \le a \delta^3 N, 
\quad \sum_{k\neq 3,4} Z_k(0) \le a \delta^3 N.
\end{equation}
Then there exists a 
stopping time $S_2$,  with $S_2 \le 4a \delta^2 N$, such that
\begin{align}
\label{behavioratS_2}
& Z_5(S_2) \ge (1 - 5a \delta^2) N, \quad Z_1(S_2)+Z_2(S_2)\le 5a\delta^2 N, \\
& \quad \,\, \,\,  Z_3(S_2) =  Z_4(S_2) = 0, 
 \quad W_6(S_2) \le 10a\delta^5 N \notag, 
\end{align}
and
\begin{equation}
\label{behaviorbetweenS_1S_2}
Z (t) \ge (1-5a\delta^2)N,  \quad \text{for all } t\in [0,S_2],
\end{equation}
all hold on a set $G_{S_2}$ with $P(G_{S_2}) \ge 1- C_{\delta}\text{e}^{-c_{\delta}N}$
for some $C_{\delta},c_{\delta}>0$.
\end{proposition}
\begin{proof}
By showing that, off of an exceptional set, the amount of work ever to be present at Station IV over 
$[0,4a\delta^2 N]$ is strictly less than 
$4a\delta^2 N$, 
it will follow that
\begin{equation}
\label{stationivbecomes0}
Z_3(t) = Z_4(t) = 0 \quad \text{for some } t\le 4a\delta^2 N,
\end{equation}
which implies the first part of the claim by setting $S_2$ equal to
the first such $t$.

By (\ref{invertedLD}),
\begin{equation}
\label{shorttimearrivalsS2}
P\big(A_1 (4a\delta^2 N) \ge 4a \delta^2 (1+\delta^2) N \big) \le C_{\delta} \text{e}^{-c_{\delta}N}
\end{equation}
for appropriate $C_{\delta}, c_{\delta} > 0$.  By (\ref{eqtimeS_2}), off of this exceptional set, the
number of jobs arriving at Class 3 over $(0,4a\delta^2 N]$ is at most
$ 4a \delta^2 \big(1+ \delta /4 +  \delta^2 \big) N$. Since
$m_3 = 1 - \delta +  \delta^3$, it follows from this that the amount of work
arriving at Class 3 over $(0,4a\delta^2 N)$ is at most 
\begin{equation}
\label{influenceatclass3S2}
4a \delta^2 \big(1+ \delta /4 + \delta^2 \big) \big(1-\delta +\delta^3 \big) N \le 
4a\delta^2  \big(1- 3\delta /4 + \delta^2 \big)N;
\end{equation}
together with
$W_3(0) \le a \delta^3 N$, this implies that the total amount of work ever
to be at Class 3 over $[0,4a\delta^2 N]$ is at most 
$4a\delta^2 (1-\delta/2 + \delta^2)N$.

The analog of (\ref{shorttimearrivalsS2}), but for Class 4, togther with
$m_4 = \delta^3$ and $Z_4(0) = N$, implies that the total amount of work ever
to be at Class 4 over $[0,4a\delta^2 N]$ is at most 
$$ \delta^3 \big(N + 4a \delta^2 (1+\delta^2) N \big).$$
Adding this bound to the bound in (\ref{influenceatclass3S2}) shows that the
 total amount of work ever to be at Station IV over $[0,4a\delta^2 N]$ is
at most 
$$4a\delta^2 \big( 1  - \delta /4 + 2\delta^2  \big) N < 4a\delta^2M,$$
which implies 
(\ref{stationivbecomes0}).

By applying (\ref{invertedLD}) to $A_1 (S_2)$
and $Z_1(0) + Z_2(0) \le a\delta^3 N$, 
$$
P(Z_1(S_2) + Z_2(S_2) \ge 5a\delta^2 N) 
\le C_{\delta}\text{e}^{-c_{\delta}N}$$
and, by applying (\ref{invertedLD}) to  $D_5 (S_2)$ and the definition of
$S_2$,
\begin{equation}
\label{Z5overS2interval}
P(Z_5(S_2) \le (1-5a\delta^2)N) \le  C_{\delta}\text{e}^{-c_{\delta}N},
\end{equation}
for some $C_{\delta},c_{\delta}>0$.  Moreover, because of
$m_1 = m_6 - m_5 = \delta^3$
 and the above two bounds on $A_1 (S_2)$ and $D_5(S_2)$, the total amount
of work ever to be at Station I over $[0,S_2]$, and hence at Class 6,  is
at most 
$$2\delta^3 \cdot 5a\delta^2 N  = 10a\delta^5$$
off an exceptional set.
Since $Z_3(S_2) = Z_4(S_2) = 0$ by the definition of $S_2$, this completes
the demonstration of (\ref{behavioratS_2}).
The same reasoning as for (\ref{Z5overS2interval}) implies
 (\ref{behaviorbetweenS_1S_2}),
which completes the proof of the proposition.

\end{proof}

The proof of Theorem \ref{propinductionstep} follows quickly from Theorem \ref{theorembehavioratS} and Proposition \ref{propatS_2} by setting 
$N = Z_4(S_1)$ and $a = 63$ in Proposition \ref{propatS_2}.
(The factor 63 is used because of the 
possible range of 9 for $Z(S_1)$ in (\ref{behavioratS_2}) and the coefficient 7 
in the
bounds on $Z_k(S_1)$.)  
%
\begin{proof}[Proof of Theorem \ref{propinductionstep}]
$\,$ Setting $T = S_1 + S_2$, the lower bound on $T$ is immediate from 
Theorem \ref{theorembehavioratS} and the upper bound also follows
because $S_2 \le N$ off of the exceptional set in Proposition \ref{propatS_2}.  The bound on $Z_5(T)$
follows from the lower bound on
$Z_4(S_1)$ in (\ref{behavioratS_1}) and the lower bound on $Z_5(S_2)$
in (\ref{behavioratS_2}).  The bound on $Z_1(T) + Z_2(T)$ follows from that
in (\ref{behavioratS_1}) with $a=63$; since Station IV is empty at time $T$, $Z_3(T) = Z_4(T) =0$.  The bounds on $W_3(S_1)$ in  (\ref{behavioratS_1})
and $W_6(S_2)$ in (\ref{behavioratS_2}) imply the bound on $W_6(T)$.
The lower bound on $Z(t)$ over $[0,T]$ follows from the corresponding bounds
in (\ref{behaviorattimesbeforeS}) and (\ref{behaviorbetweenS_1S_2}).
 
\end{proof}

{\bf Remark 4} $\,$ We commented after Theorem \ref{theorem1} that the 
analog of (\ref{firstlimit}) holds for the total amount of work in the network.
More precisely, we denote by $\mathscr{W}(t)$ the \emph{total workload} in the network at time $t$, that is, the sum of the immediate workload due to the residual service times of jobs currently at their respective classes, together with the service times of these jobs at all classes they will visit before leaving the network.  Then
\begin{equation}
\label{eqTWlimit}
\mathscr{W}(t) \rightarrow\infty \quad \text{almost surely as } t\rightarrow\infty.
\end{equation}

We sketch here the argument for (\ref{eqTWlimit}), using 
bounds from the demonstration of
Theorem \ref{propinductionstep}.
Under the assumptions in 
(\ref{eqtime0-1}) of Theorem \ref{propinductionstep}, the reasoning for (\ref{smallZbytimeS}) of Proposition \ref{corminimalocc} implies that
\begin{equation}
\label{smallZbytimeSfor2}
P\big(\text{both } Z_2(t) \le N/3 \, \text{ and } \,  Z_4(t) \le N/3 \,\, \text{ for any } t\in [0,S_1] \big) \le C_{\delta} \text{e}^{-c_{\delta}N}
\end{equation}
and the reasoning for (\ref{Z5overS2interval}) of Proposition \ref{propatS_2}
implies that
\begin{equation}
\label{Z5overS2intervalforTW}
P\big(Z_4(t) + Z_5(t) \le (1-5a\delta^2)N  \,\, \text{ for any } t\in [S_1,T] \big) \le  C_{\delta}\text{e}^{-c_{\delta}N},
\end{equation}
where, in each case, $C_{\delta},c_{\delta}>0$.  
In particular, (\ref{smallZbytimeSfor2}) was obtained by bounding
the number of jobs at Class 2, at time 0, that can leave Class 2 over $[0,N/2]$, and the number of
jobs at Class 4, at time $N/2$, that can leave Class 4 over $[N/2,S_1]$;
and (\ref{Z5overS2intervalforTW}) was obtained by
bounding the number of jobs that can leave Class 5 over $[S_1,T]$.

The service laws
at Classes 2 and 5 are each exponentially distributed with mean $1-\delta$,
and jobs currently at Class 4 must pass through Class 5 before leaving the network.
Together with
the previous paragraph and elementary large deviations
estimates, these
observations imply that
\begin{equation}
\label{stepforTW}
P(\mathscr{W}(t) \le N/6  \,\,\, \text{ for any } t\in [0,T]) \le C_{\delta} \text{e}^{-c_{\delta}N} 
\end{equation}
for appropriate $C_{\delta},c_{\delta}>0$.
The limit (\ref{eqTWlimit}) follows from (\ref{stepforTW}) and the same 
induction argument as for (\ref{eqGlimit}) in the proof
of Theorem \ref{theorem1}.

\section{Appendix}
\label{secappendix}
As mentioned in Subsection \ref{sechistorical}, the processor sharing (PS) discipline is stable for all subcritical queueing 
networks with Poisson input and exponentially distributed service times;
its equilibrium distribution can be written explicitly in its famous ``product
form". 
As with symmetric queueing networks in general, the exponentially
distributed service law can
be relaxed to a mixture of Erlang distributions.
(An Erlang distribution is the distribution of a sum of i.i.d. exponentially distributed random variables.)
This generalization employs the \emph{method of stages}
(see, e.g., \cite{Br08}, \cite{Ke79}), and the resulting equilibria are again explicit. 
It is not difficult to check that mixtures of Erlang distributions are dense in the weak topology in the set of distribution functions (see, e.g., Exercise 3.3.3 in \cite{Ke79}.)

Consider a subcritical queueing network with Poisson input, but arbitrary
service distribution.  Because of the explicit nature of the equilibria  for mixtures of Erlang distributions, one can choose a sequence of 
queueing networks whose service distributions converge to the 
service distributions of the given network
and the corresponding sequence of equilibria is tight.  Because of this, the above
representation of equilibria extends to all subcritical  PS queueing networks with Poisson input
(\cite{Ba76}).  In particular, these queueing networks with general service distributions are stable.

This explicit representation no longer holds when external arrivals 
are generalized to renewal processes.  However, for networks with 
 general interarrival distributions (assuming only (\ref{unboundedness}) and
(\ref{abscontinuity}) below)
but where the service times are exponentially distributed, one
can compare PS networks with networks with the
head-of-the-line processor sharing (HLPPS) discipline.  The service rule for
the HLPPS discipline is the same as that for PS, except that all service 
a class receives is devoted to the earliest arriving job at that class, rather
than being spread out uniformly among jobs at that class.  When
the service distributions are exponentially distributed, the specific rule
assigning service within a class does not affect the rate at which jobs leave
the class, and so processes with the PS and HLPPS disciplines and
exponentially distributed service times have the same law.

The stability of subcritical HLPPS networks can be shown 
under interarrival distributions satisfying  (\ref{unboundedness}) and
(\ref{abscontinuity}),
and general service distributions by employing the standard machinery of fluid limits
(Corollary 1 of Theorem 1, in \cite{Br96b}).  This technique unfortunately produces 
little qualitative information about the
nature of the equilibria for the corresponding queueing networks.

In this appendix, we make two observations.  First, that the 
connection between the PS and HLPPS disciplines, together with the method of stages, enables
one to quickly show, using known results, the stability of subcritical PS networks, with interarrival 
times satisfying  (\ref{unboundedness}) and
(\ref{abscontinuity}), for a dense family of service distributions.  This is
done in Proposition \ref{propPSHLPPS} below.

The second observation is that it nevertheless appears to be difficult
to extend this result to all service distributions.  One cannot
use the above argument that was applied for Poisson input without somehow first showing the tightness of the equilibria for the corresponding sequence of
queueing networks.
 Because of lack of a direct characterization of
these equilibria, it is not clear how to proceed. Nevertheless,
based on ``obvious" intuition, such stability should hold for subcritical PS
queueing networks with both general renewal input and service distributions.

For  Proposition \ref{propPSHLPPS}, we first state Corollary 1 of Theorem 1 (\cite{Br96b}) when the
service times are exponentially distributed.  Two 
assumptions on the interarrival times are required:  The interarrival time distributions are unbounded, that is, denoting  by $\xi_k (i)$, $i\in \mathbb{Z}_+$, the i.i.d. 
interarrival times at a class $k$ with external arrivals,
\begin{equation}
\label{unboundedness}
P(\xi_k(1) \ge x)  > 0 \quad \text{ for all } x.
\end{equation}
Moreover, for some $\ell_k > 0$ and non-negative $q_k(\cdot)$, with
$\int_0^{\infty}q_k(x)dx > 0$,
\begin{equation}
\label{abscontinuity}
P(\xi_k(1) + \cdots +\xi(\ell_k) \in dx) \ge q_k(x) dx,
\end{equation}
that is, the above sum dominates Lebesque measure in an 
appropriate sense.
We note that both properties are only needed to ensure that all states communicate with one another; they are not needed to show that the total number of jobs in the network $Z(t)$ is
tight (without these conditions, the residual interarrival times could synchronize in
some manner).

\begin{proposition}
\label{propHLPPS}
Any subcritical HLPPS queueing network with exponentially distributed service
times and whose external interarrival
times satisfy (\ref{unboundedness}) and (\ref{abscontinuity}) is stable.
\end{proposition}

In Section \ref{intro}, LIFO queueing networks, with routing given by
Figure \ref{fig1}, were reinterpreted as LIFO networks, with routing given by
 Figure \ref{fig2},  by decomposing the service
time at a class into  service times at successive
classes at the same station. 
Since the PS
discipline is symmetric, the same reasoning 
can be applied to it as well; in fact, since service 
is assigned uniformly
to all jobs within a station, it is easy to see that
any reclassification of classes within a station will not affect
service.  

This reasoning can be applied to service times that are mixtures of 
Erlang distributions, as in the 
method of stages.   One thus extends results on the
equilibria for subcritical PS networks with renewal external arrivals and 
exponentially distributed service times to subcritical PS networks with renewal external arrivals and service 
times that
are mixtures of Erlang distributions.  This reasoning was applied
for
Poisson external arrivals (see, e.g., \cite {Ke79}).  In the current setting,
one does not obtain an explicit formula for equilibria, but stability nevertheless
follows:
\begin{proposition}
\label{propPSHLPPS}
Any subcritical PS queueing network whose external interarrival times satisfy
 (\ref{unboundedness}) and (\ref{abscontinuity}) and whose 
service times are mixtures of Erlang distributions is stable.
\end{proposition}
\begin{proof}
By Proposition \ref{propHLPPS}, any HLPPS network whose external
interarrivals satisfy (\ref{unboundedness}) and (\ref{abscontinuity})
and whose service times are exponentially distributed is stable.  By
the above reasoning, the
same queueing network, but with the PS rather than the  HLPPS discipline,
is also stable.  On account
of the PS discipline,
the evolution of a queueing network will be the same when
classes at a given same station are combined into a single class.  
It follows from this that a subcritical PS network  whose external
interarrival times satisfy
 (\ref{unboundedness}) and (\ref{abscontinuity}) and whose 
 service times are mixtures of
Erlang distributions will be stable.

\end{proof}

\bibliography{master}     

\begin{thebibliography}{14}
\providecommand{\natexlab}[1]{#1}
\providecommand{\url}[1]{{#1}}
\providecommand{\urlprefix}{URL }
\expandafter\ifx\csname urlstyle\endcsname\relax
  \providecommand{\doi}[1]{DOI~\discretionary{}{}{}#1}\else
  \providecommand{\doi}{DOI~\discretionary{}{}{}\begingroup
  \urlstyle{rm}\Url}\fi
\providecommand{\eprint}[2][]{\url{#2}}

\bibitem[{Barbour(1976)}]{Ba76}
Barbour AD (1976) Networks of queues and the method of stages. Advances in Appl
  Probability 8(3):584--591, \doi{10.2307/1426145},
  \urlprefix\url{https://doi-org.ezp3.lib.umn.edu/10.2307/1426145}

\bibitem[{Baskett et~al.(1975)Baskett, Chandy, Muntz, and Palacios}]{BCMP75}
Baskett F, Chandy K, Muntz R, Palacios F (1975) Open, closed, and mixed
  networks of queues with different classes of customers. J ACM 22(2):248--260

\bibitem[{Bramson(1994)}]{Br94}
Bramson M (1994) Instability of {FIFO} queueing networks. The Annals of Applied
  Probability 4(2):414--431, \doi{10.1214/aoap/1177005066}

\bibitem[{Bramson(1996{\natexlab{a}})}]{Br96a}
Bramson M (1996{\natexlab{a}}) Convergence to equilibria for fluid models of
  {FIFO} queueing networks. Queueing Systems Theory Appl 22(1-2):5--45

\bibitem[{Bramson(1996{\natexlab{b}})}]{Br96b}
Bramson M (1996{\natexlab{b}}) Convergence to equilibria for fluid models of
  head-of-the-line proportional processor sharing queueing networks. Queueing
  Systems Theory Appl 23(1-4):1--26, \doi{10.1007/BF01206549},
  \urlprefix\url{https://doi-org.ezp3.lib.umn.edu/10.1007/BF01206549}

\bibitem[{Bramson(2008)}]{Br08}
Bramson M (2008) Stability of queueing networks. Probab Surv 5:169--345,
  \doi{10.1214/08-PS137}

\bibitem[{Dai(1995)}]{Da95}
Dai J (1995) On positive harris recurrence of multiclass queueing networks: a
  unified approach via fluid limit models. The Annals of Applied Probability
  5(1):49--77

\bibitem[{Kelly(1975)}]{Ke75}
Kelly F (1975) Networks of queues with customers of different types. Journal of
  Applied Probability 12(3):542--554

\bibitem[{Kelly(1979)}]{Ke79}
Kelly F (1979) Reversibility and Stochastic Networks. Wiley, Chicester

\bibitem[{Kelly(1976)}]{Ke76}
Kelly FP (1976) Networks of queues. Advances in Appl Probability 8(2):416--432,
  \doi{10.2307/1425912},
  \urlprefix\url{https://doi-org.ezp2.lib.umn.edu/10.2307/1425912}

\bibitem[{Lu and Kumar(1991)}]{LuKu91}
Lu SH, Kumar P (1991) Distributed scheduling based on due dates and buffer
  priorities. IEEE Transactions on Automatic Control 36(12):1406--1416

\bibitem[{Petrov(1975)}]{Pe75}
Petrov VV (1975) Sums of independent random variables. Springer-Verlag, New
  York-Heidelberg, translated from the Russian by A. A. Brown, Ergebnisse der
  Mathematik und ihrer Grenzgebiete, Band 82

\bibitem[{Rybko and Stolyar(1992)}]{RySt92}
Rybko A, Stolyar A (1992) Ergodicity of stochastic processes describing the
  operation of open queueing networks. Problemy Peredachi Informatsii
  28(3):3--26

\bibitem[{Seidman(1994)}]{Se94}
Seidman TI (1994) " first come, first served" can be unstable! IEEE
  Transactions on Automatic Control 39(10):2166--2171

\end{thebibliography}

\end{document}